\documentclass[final,3p]{elsarticle}
 \usepackage{graphics}
 \usepackage{graphicx}
 \usepackage{epsfig}
\usepackage{amssymb}
 \usepackage{amsthm}
 \usepackage{lineno}
 \usepackage{amsmath}
   \numberwithin{equation}{section}
\usepackage{mathrsfs}

\newtheorem{thm}{Theorem}[section]

\newtheorem{lem}[thm]{Lemma}
\newtheorem{prop}[thm]{Proposition}
\newtheorem{defn}[thm]{Definition}

\biboptions{sort&compress,square}
\begin{document}
\begin{frontmatter}
\author{Kaihua Bao}
\author{Jian Wang}
\author{Yong Wang\corref{cor2}}
\ead{wangy581@nenu.edu.cn }
\cortext[cor2]{Corresponding author (Yong Wang);}
\address{School of Mathematics and Statistics, Northeast Normal University,
Changchun, 130024, P.R.China}

\title{A local equivariant  index theorem for  sub-signature operators  }
\begin{abstract}
In this paper, we prove a local equivariant index theorem  for  sub-signature operators which generalizes
the Zhang's index theorem for  sub-signature operators.
\end{abstract}
\end{frontmatter}
\section{Introduction}
The Atiyah-Bott-Segal-Singer index formula is a generalization with group action of the Atiyah-Singer index theorem.
In \cite{BV1}, Berline and Vergne gave a heat
kernel proof of the Atiyah-Bott-Segal-Singer index formula.
In \cite{LYZ}, Lafferty, Yu and Zhang presented a simple and direct geometric proof of the Lefschetz fixed point formula for an
orientation preserving isometry on an even dimensional spin manifold by Clifford asymptotics of heat kernel.
In \cite{PW}, Ponge and Wang gave a different proof of the
equivariant index formula by the Greiner's approach of the heat kernel asymptotics.
In \cite{LM}, in order to prove family rigidity theorems, Liu and Ma proved the equivariant family index formula.
In \cite{Wa2}, Wang gave another proof of the local equivarint
index theorem for a family of Dirac operators by the Greiner's approach of the heat
kernel asymptotics. In \cite{Wa3}, using the Greiner's approach of heat kernel asymptotics, Wang
gave proofs of the equivariant Gauss-Bonnet-Chern formula and the variation
formulas for the equivariant Ray-Singer metric, which are originally due to J. M.
Bismut and W. Zhang.

 In parallel, Freed\cite{DF} considered the case of an orientation reversing involution acting on an odd
dimensional spin manifold and gave the associated Lefschetz formulas by the K-theretical way.  In \cite{Wa1},
Wang constructed an even spectral triple
by the Dirac operator and the orientation-reversing involution and computed the Chern-Connes character for this spectral triple.
In \cite{LW}, Liu and Wang proved an equivariant odd index theorem for Dirac operators with involution parity
and the Atiyah-Hirzebruch vanishing theorems for odd dimensional spin manifolds.

In \cite{Zh} and \cite{Zh1}, Zhang introduced  the sub-signature operators and  proved a local index
formula for these operators. In \cite{MZ} and \cite{DZ}, by computing the adiabatic limit of eta-invariants associated to the so-called
 sub-signature operators, a new proof of the Riemann-Roch-Grothendieck type formula of Bismut-Lott was given.
The motivation of the present article is to prove a local equivariant  index
formula for sub-signature operators.

This paper is organized as follows: In Section 2,  we present some notations of  sub-signature operators.
In Section 3.1, we prove a local equivariant even dimensional   index
formula for sub-signature operators. In Section 3.2, we prove a local equivariant odd dimensional   index
formula for sub-signature operators with the orientation-reversing involution.

\section{The sub-signature operators  }
 Firstly we give the standard setup (also see Section 1 in \cite{Zh}).
Let $M$ be an oriented closed manifold of dimension $n$. Let $E$ be an oriented sub-bundle
of the tangent vector bundle $TM$.
Let $g^{TM}$ be a metric on $TM$. Let $g^{E}$ be the induced metric on $E$. Let $E^{\perp}$ be the
sub-bundle of $TM$ orthogonal to $E$ with respect to $g^{TM}$. Let $g^{E^{\perp}}$ be the metric on $E^{\perp}$
induced from $g^{TM}$. Then $(TM, g^{TM})$ has the following orthogonal splittings
 \begin{eqnarray}
&&TM=E\oplus E^\perp, \\
&&g^{TM}=g^E\oplus g^{E^\perp}.
\end{eqnarray}
 Clearly, $E^{\perp}$ carries a canonically induced orientation. We identify the quotient bundle
$TM/E$ with $E^{\perp}$.

Let $\Omega(M)=\bigoplus_{0}^{n}\Omega^{i}(M)=\bigoplus_{0}^{n}\Gamma(\wedge^{i}(T^{*}M)) $ be the set of smooth sections of
$\wedge(T^{*}M)$. Let $*$ be the Hodge star operator of $g^{TM}$. Then $\Omega(M)$ inherits the following
standardly induced inner product
\begin{equation}
 \langle \alpha, \beta \rangle=\int_{M} \alpha\wedge \overline{*\beta},~~~\alpha, \beta\in \Omega(M).
\end{equation}

We use $g^{TM}$ to identify $TM$ and $T^{*}M$. For any $e\in\Gamma(TM)$, let $e\wedge$ and $i_{e}$ be the standard notation for
exterior and interior multiplications on $\Omega(M)$. Let $c(e)=e\wedge-i_{e}$, $\hat{c}(e)=e\wedge+i_{e}$ be the Clifford actions
on $\Omega(M)$ verifying that
 \begin{eqnarray}
&&c(e)c(e')+c(e')c(e) =-2 \langle e, e' \rangle_{g^{TM}}, \\
&&\hat{c}(e)\hat{c}(e')+\hat{c}(e')\hat{c}(e)=2\langle e, e' \rangle_{g^{TM}}, \\
&& c(e)\hat{c}(e')+\hat{c}(e') c(e)=0.
\end{eqnarray}

Denote $k = {\rm dim} E $. Let $\{f_{1}, \cdots, f_{k}\}$ be an oriented (local) orthonormal basis of $E$.
Set
\begin{equation}
\hat{c}(E, g^{E})= \hat{c}(f_{1}) \cdots \hat{c}(f_{k}),
\end{equation}
where $\hat{c}(E, g^{E})$ does not depend on the choice of the orthonormal basis.

Let
\begin{equation*}
\epsilon={\rm Id}_{\wedge^{even}(T^{*}M)}-{\rm Id}_{\wedge^{odd}(T^{*}M)}
\end{equation*}
be the $Z_{2}$-grading operator of
\begin{equation*}
\wedge(T^{*}M)=\wedge^{even}(T^{*}M)\oplus\wedge^{odd}(T^{*}M).
\end{equation*}

Set
\begin{equation}
\tau(M, g^{E})= \epsilon\hat{c}(E, g^{E}).
\end{equation}
One verifies easily that
\begin{equation}
\tau(M, g^{E})^{2}= (-1)^{\frac{k(k+1)}{2}}.
\end{equation}

Let
\begin{equation*}
\wedge_{\pm}(T^{*}M,g^{E})=\Big\{ \omega\in\wedge^{*}(T^{*}M), \tau(M, g^{E})\omega=\pm\omega \Big\}
\end{equation*}
the (even/odd) eigen-bundles of $\tau(M, g^{E})$  and by $\Omega_{\pm}(M, g^{E})$ ­§the corresponding set of smooth sections.

Let $\delta={\rm d^{*}}$ be the formal adjoint operator of the exterior differential operator ${\rm d}$ on $\Omega(M)$ with respect to the
inner product (2.3).

Set
\begin{equation}
D_{E}=\frac{1}{2}\big( \hat{c}(E, g^{E})({\rm d}+\delta)+(-1)^{k}({\rm d}+\delta)\hat{c}(E, g^{E})  \big).
\end{equation}
Then one verifies easily that
 \begin{eqnarray}
&& D_{E}\tau(M, g^{E})=- \tau(M, g^{E}) D_{E} , \\
&&D^{*}_{E}=(-1)^{\frac{k(k+1)}{2}}D_{E},
\end{eqnarray}
where $D^{*}_{E}$ is the formal adjoint operator of $D_{E}$ with respect to the inner product (2.3).

Set
\begin{equation*}
\tilde{D}_{E}=(\sqrt{-1})^{\frac{k(k+1)}{2}}D_{E}.
\end{equation*}
From (2.11), $\tilde{D}_{E}$ is a formal self-adjoint first order elliptic differential operator on $\Omega(M)$ interchanging
$\Omega_{\pm}(M,g^{E})$.

\begin{defn}
The sub-signature operator $\tilde{D}_{E,+}$ with respect to $(E, g^{TM})$ is the restriction of $\tilde{D}_{E}$ on $\Omega_{+}(M,g^{E})$.
\end{defn}
If we denote the restriction of $\tilde{D}_{E}$ on $\Omega_{-}(M,g^{E})$ by $\tilde{D}_{E,-}$, then one has clearly
\begin{equation*}
\tilde{D}^{*}_{E,\pm}=\tilde{D}_{E,\mp}.
\end{equation*}

Recall that $E$ is the subbundle of $TM$ and that we have the orthogonal decomposition  (2.1) of $TM$ and the metric $g^{TM}$. Let $P^{E}$
(resp. $P^{E^\perp}$) be the orthogonal projection from $TM$ to $E$(resp. $E^\perp$).

Let $\nabla^{TM}$ be the Levi-Civita connection of $g^{TM}$. We will use the same notation for its lifting on $\Omega(M)$.

Set
 \begin{eqnarray}
&& \nabla^{E}=P^{E}\nabla^{TM}P^{E}, \\
&& \nabla^{E^\perp}=P^{E^\perp}\nabla^{TM}P^{E^\perp}.
\end{eqnarray}
Then $\nabla^{E}$(resp.$\nabla^{E^\perp}$) is a Euclidean connection on $E$(resp.$E^\perp$), and we will use  the same notation
for its lifting on $\Omega(E^{*})$(resp. $\Omega(E^{\perp, *})$).

Let $S$ be the tensor defined by
\begin{equation*}
\nabla^{TM}=\nabla^{E}+\nabla^{E^\perp}+S.
\end{equation*}
Then $S$ takes values in skew-adjoint endomorphisms of $TM$, and interchanges $E$ and $E^\perp$.

Let $\{e_{1},\cdots,e_{n}\}$ be an oriented(local) orthonormal  base of $TM$. To specify the role of $E$,
set $\{f_{1},\cdots, f_{k}\}$ be an oriented (local) orthonormal basis of $E$. We will use the greek
subscripts for the basis of $E$. Then by Proposition 1.4 in \cite{Zh}, we have

\begin{prop}
The following identity holds,
\begin{equation}
\tilde{D}_{E}=(\sqrt{-1})^{\frac{k(k+1)}{2}}\big( \hat{c}(E, g^{E})({\rm d}+\delta)
+\frac{1}{2}\sum_{i}c(e_{i})( \nabla_{ e_{i}}^{TM}\hat{c}(E, g^{E}))  \big).
\end{equation}
\end{prop}

Similar to Lemma 1.1 in \cite{Zh}, we have

 \begin{lem}
 For any $X\in\Gamma(TM)$,  the following identity holds,
 \begin{equation}
\nabla_{X}^{TM}\hat{c}(E, g^{E})=-\hat{c}(E, g^{E})\sum_{\alpha}\hat{c}(S(X)f_{\alpha})\hat{c}(f_{\alpha}).
\end{equation}
\end{lem}

Let $\Delta^{TM}, ~~\Delta^{E}$ be the Bochner Laplacians
 \begin{eqnarray}
&&  \Delta^{TM}=\sum_{i}^{n}(\nabla_{e_{i}}^{TM,2}-\nabla^{TM}_{\nabla_{e_{i}}^{TM}e_{i}} ), \\
&&  \Delta^{E}=\sum_{i}^{k}(\nabla_{e_{i}}^{E,2}-\nabla^{E}_{\nabla_{e_{i}}^{E}e_{i}} ).
\end{eqnarray}

Let $K$ be the scalar curvature of $(M, g^{TM})$. Let $R^{TM},~ R^{E},~R^{E^\perp}$ be the curvatures of
$\nabla^{TM},~\nabla^{E},~\nabla^{E^\perp} $ respectively.
Let $\{h_{1},\cdots,h_{n-k}\}$ be an oriented(local) orthonormal  base of $E^\perp$.
 Now we can state the following Lichnerowicz type formula for $\tilde{D}^{2}_{E}$.
From Theorem 1.1 in \cite{Zh}, we have
\begin{thm}\cite{Zh}
The following identity holds,
 \begin{eqnarray}
&&\tilde{D}^{2}_{E}= - \Delta^{TM}+\frac{K}{4}
   +\frac{1}{8}\sum_{ 1\leq i,j\leq n}\sum_{ 1\leq \alpha,\beta\leq k} \langle R^{E}(e_{i},e_{j})f_{\beta}, f_{\alpha} \rangle
    c(e_{i})c(e_{j})\hat{c}( f_{\alpha})\hat{c}(f_{\beta})  \nonumber\\
&& ~~~~~+\frac{1}{8}\sum_{ 1\leq i,j\leq n}\sum_{ 1\leq s,t\leq n-k} \langle R^{E^\perp}(e_{i},e_{j})h_{t}, h_{s} \rangle
    c(e_{i})c(e_{j})\hat{c}( h_{s})\hat{c}(h_{t}) +\frac{1}{2}\sum_{ \alpha}\hat{c}\big( (\Delta^{TM}-\Delta^{E})f_{\alpha}\big)\hat{c}
    (f_{\alpha})  \nonumber\\
&&~~~~ + \sum_{i, \alpha}\Big(\hat{c}(S(e_{i})f_{\alpha})\hat{c}(f_{\alpha})\nabla^{TM}_{e_{i}}
 - \hat{c}(S(e_{i}) \nabla^{E}_{e_{i}}f_{\alpha})\hat{c}(f_{\alpha})
  + \frac{1}{2}\hat{c}\big( \nabla^{E}_{(\nabla^{TM}_{e_{i}}-\nabla^{E}_{e_{i}})e_{i}} f_{\alpha}\big)\hat{c}(f_{\alpha})
 + \frac{3}{4}\parallel S(e_{i})f_{\alpha})\parallel^{ 2}\Big)\nonumber\\
&&~~~~+ \frac{1}{4}\sum_{ i,\alpha\neq \beta}\hat{c}(S(e_{i}) f_{\alpha})\hat{c}(S(e_{i})f_{\beta})\hat{c}(f_{\alpha})\hat{c}(f_{\beta}).
\end{eqnarray}
\end{thm}

\section{A local   equivariant  index Theorem  for sub-signature operators  }
\subsection{A local even dimensional  equivariant  index Theorem  for sub-signature operators  }

Let $M$ be a closed even dimensional $n$ oriented Riemannian manifold and $\phi$ be
an isometry on $M$ preserving the orientation. Then $\phi$ induces a map
$\tilde{\phi}=\phi^{-1,*}:\wedge T_{x}^{*}M\rightarrow \wedge T_{\phi(x)}^{*}M$
on the exterior algebra bundle $\wedge T_{x}^{*}M$. Let $\tilde{D}_{E}$  be the sub-signature operator.
We assume that ${\rm d}\phi$ preserves $E$ and $E^{\perp}$ and their orientations, then $\tilde{\phi}\hat{c}(E,g^{E})=\hat{c}(E,g^{E})\tilde{\phi} $.
 Then $\tilde{\phi}\tilde{D}_{E}=\tilde{D}_{E}\tilde{\phi}$. We will compute the equivariant index
 \begin{equation}
{\rm Ind}_{\phi}(\tilde{D}_{E}^{+})={\rm Tr}(\tilde{\phi}|_{{\rm ker} \tilde{D}_{E}^{+}})-{\rm Tr}(\tilde{\phi}|_{{\rm ker} \tilde{D}_{E}^{-}}).
\end{equation}

We recall the Greiner's approach of heat kernel
asymptotics as in \cite{PW} and \cite{BeGS}, \cite{Gr}. Define
 the operator given by
  \begin{equation}
(Q_0u)(x,s)=\int_{0}^{\infty}e^{-s\tilde{D}_{E}^2}[u(x,t-s)]dt,~~u\in \Gamma_c(M\times {\mathbb{R}},\wedge T^*M),
\end{equation}
  maps continuously $u$ to $D'(M\times {\mathbb{R}},\wedge T^*M))$ which is the dual space of $\Gamma_c(M\times {\mathbb{R}},\wedge T^*M)).$
  We have
   \begin{equation}
(\tilde{D}_{E}^2+\frac{\partial}{\partial t})Q_0u=Q_0(\tilde{D}_{E}^2+\frac{\partial}{\partial t})u=u,~~~u\in \Gamma_c(M\times {\mathbb{R}},\wedge T^*M)).
\end{equation}

Let $(\tilde{D}_{E}^2+\frac{\partial}{\partial t})^{-1}$ is the Volterra inverse of $\tilde{D}_{E}^2+\frac{\partial}{\partial t}$ as in \cite{BeGS}. Then
 \begin{equation}
(\tilde{D}_{E}^2+\frac{\partial}{\partial t})Q=I-R_1;~~Q(\tilde{D}_{E}^2+\frac{\partial}{\partial t})=1-R_2,
\end{equation}
where $R_1,R_2$ are smooth operators. Let
 \begin{equation}
(Q_0u)(x,t)=\int_{M\times {\mathbb{R}}}K_{Q_0}(x,y,t-s)u(y,s)dyds,
\end{equation}
and $k_t(x,y)$ is the heat kernel of $e^{-t\tilde{D}_{E}^2}$. We get
 \begin{equation}
K_{Q_0}(x,y,t)=k_t(x,y)~ {\rm when}~ t>0,~~ {\rm when }~ t<0,~ K_{Q_0}(x,y,t)=0.
\end{equation}

 \begin{defn}
The operator $P$ is called the Volterra $\Psi DO$ if

 (i) $P$ has the Volterra property,i.e. it has a distribution
kernel of the form $K_P(x,y, t-s)$ where $K_P(x, y, t)$ vanishes on the region $t<0.$

(ii) The parabolic homogeneity of the heat operator $P +
\frac{\partial}{\partial t}$, i.e. the homogeneity with respect to
the dilations of ${\mathbb{R}}^n\times{\mathbb{R}} ^1$ given by
 \begin{equation}
\lambda\cdot (\xi,\tau)=(\lambda\xi,\lambda^2\tau),~~~~~~~(\xi,\tau)\in {\mathbb{R}}^n\times {\mathbb{R}}^1,~~\lambda\neq 0.
\end{equation}
\end{defn}

In the sequel for $g\in  \textsl{S} ({\mathbb{R}}^{n+1})$
and $\lambda\neq 0$, we let $g_{\lambda}$ be the tempered
distribution defined by
 \begin{equation}
\left<g_\lambda(\xi,\tau),u(\xi,\tau)\right>=|\lambda|^{-(n+2)}\left<g_\lambda(\xi,\tau),u(\lambda^{-1}\xi,\lambda^{-2}\tau)\right>,~~u\in
{\textsl{S}} ({\mathbb{R}}^{n+1}).
\end{equation}

 \begin{defn}
A distribution $ g\in  \textsl{S}
({\mathbb{R}}^{n+1})$ is parabolic homogeneous of degree $m,~ m \in
Z,$ if for
any $\lambda\neq 0$, we have $g_\lambda = \lambda^m g.$
\end{defn}

 Let ${\mathbb{C}}_-$ denote the complex halfplane $\{{\rm Im}\tau < 0\}$ with closure $\overline{{\mathbb{C}}_-}$. Then:

 \begin{lem}\cite{BeGS}
Let $q(\xi, \tau)\in C^{\infty}(({\mathbb{R}}^n\times {\mathbb{R}})/0)$ be a
parabolic homogeneous symbol
of degree $m$ such that:

(i) ~ $q$ extends to a continuous function on $(
{\mathbb{R}}^n\times \overline{{\mathbb{C}}_-})/0 $ in such way to
be holomorphic in the
last variable when the latter is restricted to ${{\mathbb{C}}}_-$.

 Then there is a unique $g\in  {\textsl{S}} ({\mathbb{R}}^{n+1})$ agreeing with q on ${\mathbb{R}}^{n+1}/0$ so that:

(ii)  $g$ is homogeneous of degree $m$;

(iii) The inverse Fourier transform $\breve{g}(x, t)$ vanishes for $t < 0.$
\end{lem}

 Let $U$ be an open subset of ${\mathbb{R}}^n$. We define Volterra symbols and Volterra $\Psi DO$¡¯s on $U\times {\mathbb{R}}^{n+1}/0$
as follows.

 \begin{defn}
$S_V^m(U\times
{\mathbb{R}}^{n+1}),~m\in{\mathbb{Z}}$ ,
consists in smooth functions $q(x, \xi, \tau)$ on $U\times
{\mathbb{R}}^{n}\times {\mathbb{R}}$ with
an asymptotic expansion $q\sim \sum_{j\geq 0}q_{m-j},$ where:

(i)  $q_l\in C^{\infty}(U\times [({\mathbb{R}}^{n}\times
{\mathbb{R}})/0]$ is a homogeneous Volterra symbol of degree $l$,
i.e. $q_l$ is parabolic homogeneous of degree $l$ and satisfies the
property (i) in Lemma 2.3 with respect to the last $n + 1$
variables;

(ii) The sign $\sim$ means that, for any integer $N$ and any compact $K,~ U,$ there is a constant
$C_{NK\alpha\beta k}>0$ such that for $x\in K$ and for $|\xi|+|\tau|^{\frac{1}{2}}>1$ we have
 \begin{equation}
|\partial_x^\alpha\partial_\xi^\beta\partial_\tau^k(q-\sum_{j<N}q_{m-j})(x,\xi,\tau)|\leq
C_{NK\alpha\beta k}(|\xi|+|\tau|^{\frac{1}{2}})^{m-N-|\beta|-2k}.
\end{equation}
\end{defn}

 \begin{defn}
 $\Psi_V^m(U\times
{\mathbb{R}}), ~m\in{\mathbb{Z}}$ , consists in
continuous operators $Q$ from $C_c^{\infty}(U_x\times
{\mathbb{R}}_t)$ to $C^{\infty}(U_x\times
{\mathbb{R}}_t)$
such that:

(i) $ Q $ has the Volterra property;

(ii) $Q = q(x,D_x,D_t) + R$ for some symbol $q$ in
$S_V^m(U\times {\mathbb{R}})$
and some smooth operator $R$.
\end{defn}

 In the sequel if $Q$ is a Volterra $\Psi DO$, we let $K_Q(x, y, t-s)$ denote its distribution kernel, so that
the distribution $K_Q(x, y, t)$ vanishes for $t< 0$.

 \begin{defn}
 Let $q_m(x, \xi, \tau)\in
C^{\infty}(U\times ({\mathbb{R}}^{n+1}/0))$ be a
homogeneous Volterra symbol of order $m$ and let $g_m \in
C^{\infty}(U)\otimes {\mathbb{S}}'({\mathbb{R}}^{n+1})$ denote its unique homogeneous extension given by Lemma 2.3.
Then:

(i) $\breve{q}_m(x, y, t)$ is the inverse Fourier transform of $g_m(x, \xi, \tau )$ in the last $n + 1$ variables;

(ii) $q_m(x,D_x,D_t)$ is the operator with kernel $\breve{q}_m(x, y-x, t).$
\end{defn}

 \begin{defn}
 The following properties hold.

1) Composition. Let $Q_j\in \Psi_V^{m_j}(U\times {\mathbb{R}}), ~j=1,2$
have symbol $q_j$ and suppose that $Q_1$ or $Q_2$ is
properly supported. Then $Q_1Q_2$ is a Volterra $\Psi DO$ of order $m_1+m_2$ with symbol $q_1\circ q_2\sim \sum\frac{1}{\alpha!}
\partial^\alpha_\xi
q_1D^\alpha_xq_2.$
2)  Parametrices. An operator $Q$ is the order $m$ Volterra
$\Psi DO$ with the paramatrix $P$ then
$$QP=1-R_1,~~~PQ=1-R_2\eqno(2.10)$$
where $R_1,~R_2$ are smooth operators.
\end{defn}

 \begin{defn}
 The differential operator $\tilde{D}_{E}^2 +
\partial_t$ is invertible and
its inverse $(\tilde{D}_{E}^2 + \partial_t)^{-1}$ is a Volterra $\Psi DO$ of order $-2$
\end{defn}

  We denote by $M^\phi$ the fixed-point set of $\phi$, and for $a = 0,\cdots ,n,$ we let
   $M^\phi=\bigcup _{0\leq a\leq n} M_{a}^\phi$, where  $M_{a}^\phi$ is an $a$-dimensional submanifold. Given a fixed-point $x_0$ in a component
   $M_{a}^\phi$, consider some local coordinates $x = (x^1,\cdots , x^a)$ around
$x_0.$ Setting $b = n-a,$ we may further assume that over the range
of the domain of the local coordinates there is an orthonormal frame
$e_1(x),\cdots , e_b(x)$ of $N^\phi_z$. This defines fiber
coordinates $v = (v_1, \cdots , v_b).$ Composing with the map
$(x,v)\in N^\phi(\varepsilon_0)\rightarrow {\rm exp}_x(v)$ we then
get local coordinates $x^1,\cdots,x^a,v^1,\cdots,v^b$ for $M$ near
the fixed point $x_0$. We shall refer to this type of coordinates as
{\it tubular coordinates.} Then $N^\phi(\varepsilon_0)$ is homeomorphic with a tubular neighborhood of $M^\phi$.
 Set $i_{M^{\phi}}:M^{\phi} \hookrightarrow M$ be an inclusion map. Since ${\rm d}\phi$ preserves $E$ and $E^{\perp}$,
considering  the oriented (local) orthonormal basis $\{f_{1}, \cdots, f_{k}, h_{1},\cdots,h_{n-k}\}$, set
 \begin{equation}
{\rm d}\phi_{x_{0}}=\begin{pmatrix}
  {\rm exp}(L_{1}) & 0\\
0&  {\rm exp}(L_{2})
 \end{pmatrix},
\end{equation}
where $L_{1}\in \mathfrak{s}o(k) $ and $L_{2}\in \mathfrak{s}o(n-k)$

 Let
\begin{equation}
\widehat{A}(R^{M^{\phi}})={\rm
det}^{\frac{1}{2}}\left(\frac{R^{M^{\phi}}/4\pi}{{\rm
sinh}(R^{M^{\phi}}/4\pi)}\right);~~\nu_{\phi}(R^{N^{\phi}}):={\rm
det}^{-\frac{1}{2}}(1-{\phi}^Ne^{-\frac{R^{N^{\phi}}}{2\pi}}).
\end{equation}
By (3.38) and Lemma 2.15 (2),  and Lemma 9.13 in \cite{PW}, we
get the main Theorem in this section.
\begin{thm}( Local even dimensional equivariant  index Theorem  for sub-signature operators )

Let $x_0\in M^\phi$, then
\begin{eqnarray}
 \lim_{t\rightarrow 0}{\rm Str}\left[\tilde{\phi}(x_0)K_{t}(x_0,\phi(x_0))\right]
&=& (\frac{1}{\sqrt{-1}})^{\frac{k}{2}} 2^{\frac{n}{2}}\Big\{\widehat{A}(R^{M^{\phi}})\nu_{\phi}(R^{N^{\phi}})
 i^{*}_{M^{\phi}}\Big[{\rm det}^{\frac{1}{2} } \Big({\rm cosh}\big( \frac{R^{E}}{4\pi}-\frac{L_{1}}{2}\big) \Big)\nonumber\\
&&\times {\rm det}^{\frac{1}{2} } \Big( \frac{ {\rm sinh}
(\frac{R^{E^{\perp}}}{4\pi}-\frac{L_{2}}{2})}{\frac{R^{E^{\perp}}}{4\pi}-\frac{L_{2}}{2}}\Big)  {\rm Pf} \Big(\frac{R^{E^{\perp}}}{4\pi}-\frac{L_{2}}{2}\Big)\Big]\Big\}
^{(a,0)}(x_{0}).
\end{eqnarray}
\end{thm}

Next we give a detailed proof of Theorem 3.9.

 Let $Q=( \tilde{D}_{E}^2+\partial_t)^{-1}$. For $x\in M^\phi$ and
$t>0$ set
 \begin{equation}
I_Q(x,t):=\widetilde{\phi}(x)^{-1}\int_{N_x^\phi(\varepsilon)}\phi({\rm
exp}_xv)K_Q({\rm exp}_xv,{\rm exp}_x(\phi'(x)v),t)dv.
\end{equation}
Here we use the trivialization of $\wedge(T^*M)$ about the tubular
coordinates. Using the tubular coordinates, then
 \begin{equation}
I_Q(x,t)=\int_{|v|<\varepsilon}\widetilde{\phi}(x,0)^{-1}\widetilde{\phi}(x,v)K_Q(x,v;x,\phi'(x)v;t)dv.
\end{equation}
Let
 \begin{equation}
q^{\wedge(T^*M)}_{m-j}(x,v;\xi,\nu;\tau):=\widetilde{\phi}(x,0)^{-1}\widetilde{\phi}(x,v)q_{m-j}(x,v;\xi,\nu;\tau).
\end{equation}

Recall
 \begin{prop}\cite{PW}
 Let $Q\in
\Psi_V^m(M\times {\mathbb{R}},\wedge(T^*M)),~m\in {\mathbb{Z}}.$ Uniformly on
each component $M_{a}^\phi$
 \begin{equation}
I_Q(x,t)\sim \sum_{j \geq
0}t^{-(\frac{a}{2}+[\frac{m}{2}]+1)}I_Q^j(x) ~~~~~~{\rm
as}~~t\rightarrow 0^+,
\end{equation}
where $I_Q^j(x)$ is defined by
 \begin{equation}
I_Q^{(j)}(x):=\sum_{|\alpha|\leq
m-[\frac{m}{2}]+2j}\int\frac{v^\alpha}{\alpha!}\left(\partial_v^\alpha
q^{\wedge(T^*M)}_{2[\frac{m}{2}]-2j+|\alpha|}\right)^\vee(x,0;0,(1-\phi'(x))v;1)dv.
\end{equation}
\end{prop}

Similar to Theorem 1.2 in \cite{LM} and Section 2 (d) in \cite{JMB}, we have
\begin{eqnarray}
{\rm Str}_{\tau}[\tilde{\phi}{\rm exp}(-t\tilde{D}_{E}^{2})]
&=&(\sqrt{-1})^{\frac{k}{2}}\int_{M}{\rm Str}_{\epsilon}\Big[ \hat{c}(E,g^{E})k_t(x,\phi(x))\Big]{\rm d}
x\nonumber\\
&=&(\sqrt{-1})^{\frac{k}{2}}\int_{M}{\rm Str}_{\epsilon}[\hat{c}(E,g^{E})K_{(\tilde{D}_{E}^{2}+\partial_t)^{-1}}(x,\phi(x),t)]{\rm d}
x.
\end{eqnarray}

We will compute the local index in this trivialization.

Let $(V, q)$ be a finite dimensional real vector space equipped with a quadratic form.
Let $C(V, q)$ be the associated Clifford algebra, i.e. the associative algebra generated
by $V$ with the relations $v \cdot w + w \cdot v =-2q(v,w)$ for $v,w \in V$. Let $\{e_{1},\cdots, e_{n}\}$ be the
orthomormal basis of $(V, q)$, Let $C(V, q)\hat{\otimes}C(V,-q)$ be the grading tensor product of
$C(V, q)$ and $C(V,.q)$ and $\wedge^{ *}V \hat{\otimes}\wedge^{ *}V$ be the grading tensor product of $\wedge^{ *}V$ and $\wedge^{ *}V$.
Define the symbol map:
\begin{equation}
\sigma: C(V, q)\hat{\otimes}C(V,-q)\rightarrow \wedge^{ *}V \hat{\otimes}\wedge^{ *}V;
\end{equation}
where
$\sigma(c(e_{j_{1}})\cdots c(e_{j_{l}})\otimes 1 )=e^{j_{1}}\wedge \cdots\wedge e^{j_{1}}\otimes 1$,
$\sigma(1 \otimes\hat{ c}(e_{j_{1}})\cdots \hat{ c}(e_{j_{l}}))=1 \otimes   \hat{e}^{j_{1}}\wedge \cdots\wedge \hat{e}^{j_{1}}$.
Using the interior multiplication $\iota(e_{j}):\wedge^{ *}V \rightarrow\wedge^{ *-1}V  $ and the exterior multiplication
$\varepsilon(e_{j}):\wedge^{ *}V \rightarrow\wedge^{ *+1}V  $ , we define representations of $C(V, q)$ and $C(V,-q)$ on the exterior algebra:
\begin{eqnarray}
&&c:C(V,q)\rightarrow {\rm End}\wedge V, ~e_{j}\mapsto c(e_{j}):~\varepsilon(e_{j})-\iota(e_{j}); \\
&&c:C(V,-q)\rightarrow {\rm End}\wedge V, ~e_{j}\mapsto \hat{c}(e_{j}):~\varepsilon(e_{j})+\iota(e_{j});
\end{eqnarray}
The tensor product of these representations yields an isomorphism of superalgebras
\begin{equation}
c\otimes \hat{c}:C(V, q)\hat{\otimes} C(V,-q)\rightarrow {\rm End}\wedge V
\end{equation}
which we will also denote by $c$. We obtain a supertrace (i.e. a linear functional vanishing
on supercommutators) on $C(V, q)\hat{\otimes} C(V,-q)$ by setting ${\rm Str}(a) ={\rm Str}_{{\rm End}\wedge V} [c(a)]$
for $a\in C(V, q)\hat{\otimes} (V,-q)$, where ${\rm Str}_{{\rm End}\wedge V}$ is the canonical supertrace on ${\rm End}V$ .

\begin{lem}
For $1\leq i_{1}<\cdots < i_{p}\leq n $,$1\leq j_{1}<\cdots < j_{q}\leq n $, when $p=q=n$,
\begin{equation}
{\rm Str}[c(e_{i_1})\cdots c(e_{i_n})\hat{c}(e_{i_1})\cdots \hat{c}(e_{i_n})]=(-1)^{\frac{n(n+1)}{2}}2^{n}
\end{equation}
and otherwise equals zero.
\end{lem}

We will also denote the volume element in  $\wedge V\hat{\otimes}\wedge V$ by
$ \omega= e^{1}\wedge \cdots \wedge e^{n} \wedge\hat{e}^{1}\wedge \cdots \wedge \hat{e}^{n}$.
For $a\in \wedge V\hat{\otimes}\wedge V$, let ${\rm T}a$ be the coefficient of $\omega$. The linear functional
${\rm T}:\wedge V\hat{\otimes}\wedge V \rightarrow {\rm R}$ is called the Berezin trace. Then for a $a\in C(V, q)\hat{\otimes}(V,.q)$ , one
has ${\rm Str}_{s}(a)=(-1)^{\frac{n(n+1)}{2}}2^{n}({\rm T}\sigma)(a)$.  We define the Getzler order as follows:
\begin{equation}
{\rm deg} \partial_{ j}=\frac{1}{2} {\rm deg}\partial_{ t}=- {\rm deg} x^{ j}=1,~~~{\rm deg}c(e_{j})=1,~~{\rm deg}\hat{c}(e_{j})=0.
\end{equation}

 Let $Q\in \Psi_V^*({\mathbb{R}}^n\times {\mathbb{R}}, \wedge^*T^*M)$ have symbol
 \begin{equation}
q(x,\xi,\tau)\sim  \sum_{k\leq m'}q_{k}(x,\xi,\tau),
\end{equation}
 where $q_{k}(x,\xi,\tau)$ is an order $k$ symbol. Then taking components
in each subspace $\wedge^jT^*M\otimes\wedge^lT^*M$ of $\wedge T^*M\otimes \wedge T^*M$
and using Taylor expansions at $x = 0$ give formal expansions
\begin{equation}
\sigma[q(x,\xi,\tau)]\sim\sum_{j,k}\sigma[q_{k}(x,\xi,\tau)]^{(j,l)}\sim\sum_{j,k,\alpha}\frac{x^\alpha}{\alpha!}
\sigma[\partial_x^\alpha q_{k}(0,\xi,\tau)]^{(j,l)}.
\end{equation}
The symbol
$\frac{x^\alpha}{\alpha!} \sigma[\partial_x^\alpha
q_{k}(0,\xi,\tau)]^{(j,l)}$ is the Getzler homogeneous
of $k+j-|\alpha|$. So we can expand $\sigma[q(x,\xi,\tau)]$ as
\begin{equation}
\sigma[q(x,\xi,\tau)]\sim \sum_{j\geq 0}q_{(m-j)}(x,\xi,\tau),~~~~~~~~~q_{(m)}\neq 0,
\end{equation}
where $q_{(m-j)}$ is a Getzler homogeneous symbol of degree $m-j$.

\begin{defn}
The integer $m$ is called as the Getzler order of $Q$. The symbol $q_{(m)}$ is the principle Getzler
homogeneous symbol of $Q$. The operator $Q_{(m)}=q_{(m)}(x,D_x,D_t)$ is called as the model operator of $Q$.
\end{defn}

 Let $e_1, \dots , e_n$ be an oriented orthonormal basis of $T_{x_0}M$ such that $e_1,\cdots , e_a$ span $T_{x_0}M^{\phi}$ and
$e_{a+1},\cdots , e_n$ span $N_{ x_0}^{\phi}$ . This provides us with normal coordinates $(x_1, \cdots , x_n)\rightarrow {\rm
exp}_{x_0}(x^1e_1+\cdots+x^ne_n).$ Moreover using parallel translation enables us to construct a synchronous local oriented
tangent frame $e_1(x), . . . , e_n(x)$ such that $e_1(x),\cdots , e_a(x)$ form an oriented frame of $TM_{a}^{\phi}$ and $e_{a+1}(x),
\cdots , e_n(x)$ form an (oriented) frame $N^{\tau}$ (when both frames are restricted to $M^{\phi}).$ This gives rise to trivializations of
the tangent and exterior algebra  bundles. Write
\begin{equation}
\phi'(0)=\left(\begin{array}{lcr}
  1  & 0  \\
   0  &  \phi^N
\end{array}\right)={\rm exp}(A_{ij}),
\end{equation}
where $A_{ij}\in \mathfrak{s}o(n) $.

Let $\wedge(n)=\wedge^*{\mathbb{R}}^n$ be the  exterior algebra of ${\mathbb{R}}^n$. We shall use the
following gradings on $\wedge(n) \hat{\otimes}\wedge(n)$,
\begin{equation}
\wedge(n)\hat{\otimes}\wedge(n)=\bigoplus_{\begin{array}{lcr}
  1\leq k_{1}, k_{2}\leq a \\
  1\leq \overline{l}_{1}, \overline{l}_{2}\leq b
\end{array}}\wedge^{k_{1},\overline{l}_{1}}(n) \hat{\otimes}\wedge^{k_{2},\overline{l}_{2}}(n),
\end{equation}
 where $\wedge^{k,\overline{l}}(n)$ is the space of forms $dx^{i_1}\wedge\cdots\wedge
dx^{i_{k+\overline{l}}}$ with $1\leq i_1<\cdots <i_k\leq a$ and $a + 1\leq i_{k+1} < \cdots < i_{k+\overline{l}}\leq n.$
 Given a form
$\omega\in\wedge(n)\hat{\otimes}\wedge(n)$, denote by $\omega^{(k_{1},\overline{l}_{1}),(k_{2},\overline{l}_{2})} $
its component in $\wedge(n)^{(k_{1},\overline{l}_{1})}\hat{\otimes}\wedge^{(k_{2},\overline{l}_{2})}(n) $.
We denote by $|\omega|^{(a,0),(a,0)}$ the Berezin integral $|\omega^{(*,0),(*,0)}|^{(a,0),(a,0)}$
of its component $\omega^{(*,0),(*,0)}$ in $\wedge^{(*,0),(*,0)}(n)$.

Let $A\in Cl (V,q)\hat{\otimes}Cl (V,-q)$, then
\begin{eqnarray}
{\rm Str}[\widetilde{\phi}A]&=&(-1)^{\frac{n}{2}}2^{n}\sum _{0\leq l_{2}\leq b}|\sigma(\widetilde{\phi})^{((0,b),(0,l_{2}))}
 \sigma(A)^{((a,0),(a,b-l_{2}))}|^{(n,n)} \nonumber\\
&&+(-1)^{\frac{n}{2}}2^{n}\sum_{0\leq l_{1}<b,0\leq l_{2}\leq b}
|\sigma(\widetilde{\phi})^{((0,l_{1}),(0,l_{2}))}\sigma(A)^{((a,b-l_{1}),(a,b-l_{2}))}|^{(n,n)}.
\end{eqnarray}

In order to calculate ${\rm Str}[\widetilde{\phi}A]$, we need to consider the representation of
$ |\sigma(\widetilde{\phi})^{((0,b),(0,l_{2}))}
 \sigma(A)^{((a,0),(a,b-l_{2}))}|^{(n,n)}$.

 Let the matrix $\phi^{N}$ equal
  \begin{equation}
\phi^{N}=\begin{pmatrix}
  A_{\frac{a}{2}+1} \\
&\ddots &
& \text{{\huge{0}}}\\
& &  \ddots\\
 & \text{{\huge{0}}} & & \ddots\\
& & & &  A_{\frac{n}{2}}
 \end{pmatrix},
 ~  A_{\frac{a}{2}+1}=\left(\begin{array}{lcr}
{\rm cos }\theta_{\frac{a}{2}+1} & {\rm sin }\theta_{\frac{a}{2}+1} \\
   {\rm -sin }\theta_{\frac{a}{2}+1} &  {\rm cos }\theta_{\frac{a}{2}+1}
\end{array}\right),
~ A_{\frac{n}{2}}=\left(\begin{array}{lcr}
 {\rm cos }\theta_{\frac{n}{2}} & {\rm sin }\theta_{\frac{n}{2}}  \\
   {\rm -sin }\theta_{\frac{n}{2}} &  {\rm cos }\theta_{\frac{n}{2}}
\end{array}\right).
 \end{equation}
From Lemma 3.2 in \cite{ZJW}  , we have

\begin{lem}
 \begin{eqnarray}
\tilde{\phi}&=&(\frac{1}{2})^{\frac{n-a}{2}}\prod_{ j=\frac{a}{2}+1}^{n}\Big[(1+{\rm cos }\theta_{j})
-(1-{\rm cos }\theta_{j})c(e_{2j-1})c(e_{2j})\hat{c}(e_{2j-1})\hat{c}(e_{2j})\nonumber\\
&&+{\rm sin }\theta_{j}\big(c(e_{2j-1})c(e_{2j})-\hat{c}(e_{2j-1})\hat{c}(e_{2j})\big)
\Big]
\end{eqnarray}
\end{lem}
Then we obtain
 \begin{eqnarray}
\sigma(\tilde{\phi})^{((0,b),(0,l_{2}))}&=&(\frac{1}{2})^{\frac{n-a}{2}}\sigma\Big\{\prod_{ j=\frac{a}{2}+1}^{n}\Big[
-(1-{\rm cos }\theta_{j})c(e_{2j-1})c(e_{2j})\hat{c}(e_{2j-1})\hat{c}(e_{2j})\nonumber\\
&&+{\rm sin }\theta_{j}\big(c(e_{2j-1})c(e_{2j})\big)\Big]\Big\}^{((0,b),(0,l_{2}))}\nonumber\\
&=&(\frac{1}{2})^{\frac{n-a}{2}}e^{a+1}\wedge\cdots\wedge e^{n}\sigma\Big\{ \prod_{ j=\frac{a}{2}+1}^{n}
\Big[
-(1-{\rm cos }\theta_{j})\hat{c}(e_{2j-1})\hat{c}(e_{2j})
+{\rm sin }\theta_{j}\Big] \Big\}^{(0,l_{2})}\nonumber\\
&=&(\frac{1}{2})^{\frac{n-a}{2}}e^{a+1}\wedge\cdots\wedge e^{n} \sigma\Big\{\prod_{ j=\frac{a}{2}+1}^{n}2{\rm sin }\frac{\theta_{j}}{2}
\Big[{\rm cos }\frac{\theta_{j}}{2}
-{\rm sin }\frac{\theta_{j}}{2}\hat{c}(e_{2j-1})\hat{c}(e_{2j})\Big]\Big\}^{(0,l_{2})} \nonumber\\
&=&(\frac{1}{2})^{\frac{n-a}{2}}e^{a+1}\wedge\cdots\wedge e^{n}{\rm det }^{\frac{1}{2}}(1-\phi^{N})
\sigma\Big[{\rm exp}\big( -\frac{1}{4}\sum_{ 1\leq i,j \leq n} A_{ij}\hat{c}(e_{i})\hat{c}(e_{j}) \big)\Big]^{(0,l_{2})}\nonumber\\
&=&(\frac{1}{2})^{\frac{n-a}{2}}e^{a+1}\wedge\cdots\wedge e^{n}{\rm det }^{\frac{1}{2}}(1-\phi^{N})
\sigma\Big[{\rm exp}\big(-\frac{1}{4}\sum_{ 1\leq i,j\leq k}  (L_{1})_{ij}
  \hat{c}( f_{i})\hat{c}(f_{j})\nonumber\\
&&-\frac{1}{4}\sum_{ 1\leq i,j\leq n-k} (L_{2})_{k+i,k+j}
   \hat{c}( h_{i})\hat{c}(h_{j}) \big)\Big]^{(0,l_{2})}.
\end{eqnarray}

Next we  calculate $ |\sigma(A)|^{((a,0),(a,b-l_{2}))}$.
In the sequel, we shall use the following ``curvature forms": $R':=(R_{i,j})_{1\leq i,j\leq a}$, $R'':=(R_{a+i,a+j})_{1\leq i,j\leq b}$.
Let
 \begin{eqnarray*}
&&\dot{R}=\frac{1}{4}\sum_{ 1\leq \alpha,\beta\leq k} \langle R^{E}
f_{\alpha},  f_{\beta}\rangle   \hat{c}( f_{\alpha})\hat{c}(f_{\beta}),\nonumber\\
&&\ddot{R}=\frac{1}{4}\sum_{ 1\leq s,t\leq n-k} \langle R^{E^\perp}h_{s}, h_{t} \rangle
   \hat{c}( h_{s})\hat{c}(h_{t});
\end{eqnarray*}
and
 \begin{eqnarray*}
&&\tilde{\dot{R}}=\frac{1}{4}\sum_{ 1\leq \alpha,\beta\leq k} \langle (R^{E}-L_{1})
f_{\alpha},  f_{\beta}\rangle   \hat{c}( f_{\alpha})\hat{c}(f_{\beta}),\nonumber\\
&&\tilde{\ddot{R}}=\frac{1}{4}\sum_{ 1\leq s,t\leq n-k} \langle (R^{E^\perp}-L_{2})h_{s}, h_{t} \rangle
   \hat{c}( h_{s})\hat{c}(h_{t}).
\end{eqnarray*}

By (2.19), we get

\begin{prop}
The model operator of $F$ is
\begin{eqnarray}
F_{(2)}&=&-\sum_{r=1}^n\big(\partial_r+\frac{1}{8}\sum_{1\leq i,j,l\leq n}  \langle R^{TM}(e_{i},e_{j})e_{l}, e_{r} \rangle
        y_{l} e^{i}\wedge e^{j}\big)^2 \nonumber\\
   &&+\frac{1}{8}\sum_{ 1\leq i,j\leq n}\sum_{ 1\leq \alpha,\beta\leq k} \langle R^{E}(e_{i},e_{j})f_{\beta}, f_{\alpha} \rangle
    e^{i}\wedge e^{j}\hat{c}( f_{\alpha})\hat{c}(f_{\beta})  \nonumber\\
&& +\frac{1}{8}\sum_{ 1\leq i,j\leq n}\sum_{ 1\leq s,t\leq n-k} \langle R^{E^\perp}(e_{i},e_{j})h_{t}, h_{s} \rangle
    e^{i}\wedge e^{j}\hat{c}( h_{s})\hat{c}(h_{t}).
\end{eqnarray}
\end{prop}

From the representation of $F_{(2)}$,   we get the model operator of $ \frac{\partial}{\partial t}+\tilde{D}_{E}^{2}$
is $\frac{\partial}{\partial t}+F_{(2)}$. And we have
 \begin{equation}
(\frac{\partial}{\partial t}+F_{(2)})K_{Q_{(-2)}}(x,y,t)=0.
\end{equation}

 Similar to Lemma 2.9 in \cite{PW}, we get
 \begin{lem}
Let $Q\in \Psi^{(-2)}({\mathbb{R}}^{n}\times {\mathbb{R}},,\wedge(T^*M)) $ be a parametrix for $(F_{(2)}+\partial_t)^{-1}$.
 Then

 (1)Q has Getzler order -2 and its model operator is $(F_{(2)}+\partial_t)^{-1}$;

(2) For all $t > 0$,  {\rm }
\begin{eqnarray}
&&(\sqrt{-1})^{\frac{k}{2}}\hat{c}(E,g^{E})I_{(F_{(2)}+\partial_t)^{-1}}(0,t) \nonumber\\
&=&(\sqrt{-1})^{\frac{k}{2}}\hat{c}(E,g^{E})\frac{(4\pi t)^{-\frac{a}{2}}}{ {\rm det}^{\frac{1}{2}}(1-\phi^{N})}{\rm det}^{\frac{1}{2}}
 \Big( \frac{\frac{tR'}{2}}{{\rm sinh}(\frac{tR'}{2})}\Big){\rm det}^{-\frac{1}{2}}(1-\phi^{N}e^{-tR''}){\rm exp}\big(t(\tilde{\dot{R}}
 +\tilde{\ddot{R}})\big).
\end{eqnarray}
\end{lem}

Similar to Lemma 3.6 in [Wa]. we have

\begin{lem}
 $Q\in \Psi_V^*({\mathbb{R}}^n\times
{\mathbb{R}},\wedge(T^*M))$ has the Getzler order $m$
and model operator $Q_{(m)}$. Then as $t\rightarrow 0^+$

 (1) $\sigma[I_Q(0,t)]^{(j, l)}=O(t^{\frac{j-m-a-1}{2}})$ , if
$m-j$ is odd.

(2)
$\sigma[I_Q(0,t)]^{(j,l)}=O(t^{\frac{j-m-a-2}{2}})I_{Q(m)}(0,1)^{(j,l)}+O(t^{\frac{j-m-a}{2}})$,  if
$m-j$ is even.

In particular, for $m=-2$ and $j=a$ and $a$ is even we get
 \begin{equation}
\sigma[I_Q(0,t)]^{((a,0),(a,b-l_{2}))}=I_{Q(-2)}(0,1)^{((a,0),(a,b-l_{2}))}+O(t^{\frac{1}{2}}).
\end{equation}
\end{lem}

With all these preparations, we are going to prove the local even dimensional
 equivariant index theorem  for sub-signature operators.
Substituting (3.33), (3.36) into (3.30), we obtain

\begin{eqnarray}
&& \lim_{t\rightarrow 0}{\rm Str}_\varepsilon\left[\tilde{\phi}(x_0) (\sqrt{-1})^{\frac{k}{2}}\hat{c}(E,g^{E}) I_{(F+\partial_t)^{-1}}(x_0,t)\right]\nonumber\\
&=& (-1)^{\frac{n}{2}}2^{n}(\frac{1}{2})^{\frac{n-a}{2}}(4\pi)^{-\frac{a}{2} }(\sqrt{-1})^{\frac{k}{2}}
\big| \widehat{A}(R^{M^{\phi}})\nu_{\phi}(R^{N^{\phi}}) \sigma\big[\hat{c}( f_{1})\cdots \hat{c}( f_{k})
{\rm exp}(\tilde{\dot{R}}+\tilde{\ddot{R}})\big] \big |^{((a,0),n)}
 \nonumber\\
&=&(\frac{1}{\sqrt{-1}})^{\frac{k}{2}} 2^{\frac{n}{2}}\Big\{\widehat{A}(R^{M^{\phi}})\nu_{\phi}(R^{N^{\phi}})
 i^{*}_{M^{\phi}}\Big[{\rm det}^{\frac{1}{2} } \Big({\rm cosh}\big( \frac{R^{E}}{4\pi}-\frac{L_{1}}{2}\big) \Big)\nonumber\\
&&\times {\rm det}^{\frac{1}{2} } \Big( \frac{ {\rm sinh}
(\frac{R^{E^{\perp}}}{4\pi}-\frac{L_{2}}{2})}{\frac{R^{E^{\perp}}}{4\pi}-\frac{L_{2}}{2}}\Big)  {\rm Pf} \Big(\frac{R^{E^{\perp}}}{4\pi}-\frac{L_{2}}{2}\Big)\Big]\Big\}
^{(a,0)}(x_{0}).
\end{eqnarray}
Where  we have used the algebraic result of Proposition 3.13 in \cite{BGV} , and
the Berezin integral in the right hand side of (3.38) is  the application of the following lemma.
 \begin{lem}
\begin{eqnarray}
&&|\sigma\big[\hat{c}( f_{1})\cdots \hat{c}( f_{k}){\rm exp}(\tilde{\dot{R}}+\tilde{\ddot{R}})\big]|^{(n)}\nonumber\\
&=& (-1)^{\frac{n-k}{2}}
{\rm det}^{\frac{1}{2} } \Big({\rm cosh}\big( \frac{R^{E}-L_{1}}{2}\big) \Big)
 {\rm det}^{\frac{1}{2} } \Big( \frac{ {\rm sinh}
(\frac{R^{E^{\perp}}-L_{2}}{2})}{(R^{E^{\perp}}-L_{2})/2}\Big)  {\rm Pf} \Big(\frac{R^{E^{\perp}}-L_{2}}{2}\Big)
.
\end{eqnarray}
\end{lem}
 \begin{proof}
  In order to compute this differential form, we make use of  the Chern root algorithm (see [9]).
   Assume that $n={\rm dim}M$ and  $k={\rm dim}E$ are both even integers.
As in [5], we write
\begin{equation}
R^{E}-L_{1}=\begin{pmatrix}
  \begin{pmatrix}
0 & -\theta_{1} \\
\theta _{1} & 0
 \end{pmatrix}
 & &   & \text{{\huge{0}}}\\
& &  \ddots\\
 \text{{\huge{0}}}
 &  &  &
   \begin{pmatrix}
    0 & -\theta _{-\frac{k}{2}}  \\
    \theta _{-\frac{k}{2}}& 0
   \end{pmatrix}
 \end{pmatrix},
R^{E^\perp}-L_{2}=\begin{pmatrix}
  \begin{pmatrix}
0& -\hat{\theta} _{1} \\
\hat{\theta} _{1} & 0
 \end{pmatrix}
 & &   & \text{{\huge{0}}}\\
& &  \ddots\\
 \text{{\huge{0}}}
 &  &  &
   \begin{pmatrix}
    0 & -\hat{\theta} _{\frac{n-k}{2}}  \\
   \hat{\theta} _{\frac{n-k}{2}} & 0
   \end{pmatrix}
 \end{pmatrix}
\end{equation}

Then we obtain
 \begin{eqnarray}
\frac{1}{4}\sum_{ 1\leq \alpha,\beta\leq k} \langle (R^{E}-L_{1})
f_{\alpha},  f_{\beta}\rangle   \hat{c}( f_{\alpha})\hat{c}(f_{\beta})
&=&\frac{1}{2}\sum_{ 1\leq \alpha<\beta\leq k} \langle (R^{E}-L_{1})
f_{\alpha},  f_{\beta}\rangle   \hat{c}( f_{\alpha})\hat{c}(f_{\beta})\nonumber\\
&=&\frac{1}{2}\sum_{ 1\leq j\leq \frac{k}{2}}  \theta _{j}   \hat{c}( f_{2j-1})\hat{c}(f_{2j});
\end{eqnarray}

 \begin{eqnarray}
\frac{1}{4}\sum_{ 1\leq s,t\leq n-k} \langle (R^{E^\perp}-L_{2})h_{s}, h_{t} \rangle
   \hat{c}( h_{s})\hat{c}(h_{t})
   &=&\frac{1}{2}\sum_{ 1\leq s<t\leq n-k} \langle (R^{E^\perp}-L_{2})h_{s}, h_{t} \rangle
   \hat{c}( h_{s})\hat{c}(h_{t})\nonumber\\
&=&\frac{1}{2}\sum_{ 1\leq l\leq \frac{n-k}{2}} \hat{ \theta} _{l}   \hat{c}( h_{2l-1})\hat{c}(h_{2l}).
\end{eqnarray}

Then the  left hand side of (3.39) is
  \begin{eqnarray}
 &&\Big |\sigma\Big(\hat{c}( f_{1})\cdots \hat{c}( f_{k}){\rm exp}(\tilde{\dot{R}}+\tilde{\ddot{R}})\Big)\Big |^{(n)}\nonumber\\
&=& \Big |\sigma\Big(\hat{c}( f_{1})\cdots \hat{c}( f_{k})
\prod_{1\leq j\leq \frac{k}{2}}{\rm exp}(\frac{1}{2}  \theta _{j}
\hat{c}( f_{2j-1})\hat{c}(f_{2j}))
  \prod_{1\leq l\leq \frac{n-k}{2}} {\rm exp}( \frac{1}{2} \hat{ \theta} _{l} \hat{c}( h_{2l-1})\hat{c}(h_{2l}))\Big)\Big |^{(n)}\nonumber\\
&=& \Big |\sigma\Big(\hat{c}( f_{1})\cdots \hat{c}( f_{k})
\prod_{1\leq j\leq \frac{k}{2}}
\Big[{\rm cos }\frac{\theta_{j}}{2}
-{\rm sin }\frac{\theta_{j}}{2} \hat{c}( f_{2j-1})\hat{c}(f_{2j})\Big]
 \prod_{1\leq l\leq \frac{n-k}{2}}  \Big[{\rm cos }\frac{\hat{\theta_{l}}}{2}
-{\rm sin }\frac{\hat{\theta_{l}}}{2}\hat{c}( h_{2l-1})\hat{c}(h_{2l})\Big]\Big)\Big |^{(n)}\nonumber\\
&=&(-1)^{\frac{n-k}{2}}\prod_{1\leq j\leq \frac{k}{2}} {\rm cos }\frac{\theta_{j}}{2}  \prod_{1\leq l\leq \frac{n-k}{2}}
 {\rm sin }\frac{\hat{\theta}_{l}}{2}.
\end{eqnarray}

 Now we consider the right hand side of (3.39),

 \begin{equation}
\big(R^{E}-L_{1}\big)^{2p}=(-1)^{p}\begin{pmatrix}
  \begin{pmatrix}
\theta _{1}^{2p} & 0 \\
0 & \theta _{1}^{2p}
 \end{pmatrix}
 & &   & \text{{\huge{0}}}\\
& &  \ddots\\
 \text{{\huge{0}}}
 &  &  &
   \begin{pmatrix}
   \theta _{\frac{k}{2}}^{2p} & 0 \\
   0  & \theta _{\frac{k}{2}}^{2p}
   \end{pmatrix}
 \end{pmatrix},
\end{equation}
  Then
  \begin{equation}
 {\rm det}^{\frac{1}{2} } \Big({\rm cosh}\big( \frac{R^{E}-L_{1}}{2}\big) \Big)
 =\prod_{ j=1 }^{\frac{k}{2}}\Big( \sum_{ p=0}^{\infty} \big(\frac{\theta_{j}}{2}\big)^{2p}\frac{(-1)^{p}}{(2p)!}\Big)
  =\prod_{ j=1 }^{\frac{k}{2}}{\rm cosh}\frac{\sqrt{-1}\theta_{j}}{2}
  =\prod_{ j=1 }^{\frac{k}{2}} \frac{e^{\frac{\sqrt{-1}\theta_{j}}{2}}+e^{-\frac{\sqrt{-1}\theta_{j}}{2}}}{2}
  =\prod_{ j=1 }^{\frac{k}{2}}{\rm cos}\frac{\theta_{j}}{2}.
 \end{equation}
 Similarly, we have
   \begin{equation}
{\rm det}^{\frac{1}{2} } \Big( \frac{ {\rm sinh}(\frac{R^{E^\perp}-L_{2}}{2})}{(R^{E^\perp}-L_{2})/2}\Big)
  =\prod_{ j=1 }^{\frac{n-k}{2}}\frac{{\rm sin}\frac{\hat{\theta}_{j}}{2}}{\frac{\hat{\theta}_{j}}{2}}.
 \end{equation}
On the other hand,
   \begin{equation}
{\rm Pf} (\frac{R^{E^\perp}-L_{2}}{2})=T\Big({\rm exp}\big(\sum_{s<t}\langle\frac{ R^{E^\perp}-L_{2}}{2}h_{s},  h_{t}\rangle
h^{s}\wedge h^{t}\big)   \Big)
=T\Big({\rm exp}\big(\sum_{1\leq j\leq \frac{n-k}{2}}\frac{\hat{\theta}_{j}}{2} h^{2j-1}\wedge h^{2j}  \big) \Big)
=\prod_{ j=1 }^{\frac{n-k}{2}}\frac{\hat{\theta}_{j}}{2}.
 \end{equation}
 Combining these equations,  the  proof of lemma 3.19 is completed.
\end{proof}

In summary, we have proved Theorem 3.9.

\subsection{The local odd dimensional equivariant index Theorem for  sub-signature operators}

In this section, we give a proof of a  local odd dimensional  equivariant index theorem  for sub-signature operators.
Let $M$ be an odd dimensional oriented closed   Riemannian manifold. Using (2.19) in Section 2, we may define the
sub-signature operators $  \tilde{D}_{E}$. Let $\gamma$  be an orientation reversing involution  isometric acting on $M$.
 Set ${\rm d}\gamma$ preserves $E$ and $E^{\perp}$ and  preserves the orientation of $E$ ,  then $\tilde{\gamma}\hat{\tau}(E,g^{E})=\hat{\tau}(E,g^{E})\tilde{\gamma}$,
where $\tilde{\gamma}$ is the lift
on the exterior algebra bundle $\wedge T^{*}M$ of $\gamma$.
There exists a self-adjoint lift
$\widetilde{\gamma}:~\Gamma(M; \wedge(T^*M))\rightarrow \Gamma(M; \wedge(T^*M))$ of $\gamma$ satisfying
\begin{equation}
\widetilde{\gamma}^2=1;~~  \tilde{D}_{E}\widetilde{\gamma}=-\widetilde{\gamma}  \tilde{D}_{E}.
\end{equation}
 Now the $+1$ and $-1$
eigenspaces of $\widetilde{\gamma}$ give a splitting
\begin{equation}
\Gamma(M;\wedge(T^*M))\cong \Gamma^+(M; \wedge(T^*M)){\small \oplus} \Gamma^{-}(M; \wedge(T^*M)))
\end{equation}
then the  sub-signature operator interchanges $\Gamma^+(M; \wedge(T^*M))$ and $\Gamma^-(M; \wedge(T^*M))$,
and $\hat{c}(E,g^{E})$ preserves  $\Gamma^+(M; \wedge(T^*M))$ and $\Gamma^-(M; \wedge(T^*M))$.

Denotes by $\tilde{D}_{E}^{+}$ the restriction of
$\tilde{D}_{E}$ to $\Gamma^+(M, \wedge(T^*M))$.
We assume ${\rm dim}E=k$ is even, then $(\tilde{D}_{E})\hat{c}(E,g^{E})=\hat{c}(E,g^{E})(\tilde{D}_{E})$ and $\hat{c}(E,g^{E})$ is a
linear map from ${\rm ker}   \tilde{D}_{E}^{\pm}$ to ${\rm ker}  \tilde{D}_{E}^{\pm}$.

The purpose of this section is to compute
\begin{equation}
{\rm ind}_{\hat{c}(E,g^{E})}[(\tilde{D}^+_{E})]={\rm Tr}(\hat{c}(E,g^{E})|_{{\rm ker}   \tilde{D}_{E}^{+}})-{\rm Tr}(\hat{c}(E,g^{E})|_{{\rm ker}   \tilde{D}_{E}^{+}}).
\end{equation}
By the Mckean-Singer formular, we have

\begin{eqnarray}
{\rm ind}_{\hat{c}(E,g^{E})}(\tilde{D}^+_{E})
&=&\int_{M}
(\sqrt{-1})^{\frac{k}{2}}{\rm Tr}[\tilde{\gamma}\hat{c}(E,g^{E}) k_t(x,\gamma(x))]\nonumber dx\\
&=&\int_{M}(\sqrt{-1})^{\frac{k}{2}}{\rm Tr}[\tilde{\gamma}\hat{c}(E,g^{E}) K_{(F+\partial_t)^{-1}}(x,\gamma(x),t)]dx.
\end{eqnarray}

Let
\begin{equation}
R^{E}-L_{1}=\begin{pmatrix}
  \begin{pmatrix}
0 & -\theta_{1} \\
\theta _{1} & 0
 \end{pmatrix}
 & &   & \text{{\huge{0}}}\\
& &  \ddots\\
 \text{{\huge{0}}}
 &  &  &
   \begin{pmatrix}
    0 & -\theta _{-\frac{k}{2}}  \\
    \theta _{-\frac{k}{2}}& 0
   \end{pmatrix}
 \end{pmatrix},
R^{E^\bot}-L_{2}=\begin{pmatrix}
  \begin{pmatrix}
0 & -\hat{\theta}_{1} \\
\hat{\theta} _{1} & 0
 \end{pmatrix}
 & &   & \text{{\huge{0}}}\\
& &  \ddots\\
 \text{{\huge{0}}}
 &  &  &
   \begin{pmatrix}
    0 & -\hat{\theta} _{\frac{n-k-1}{2}}  \\
    \hat{\theta}_{\frac{n-k-1}{2}}& 0
   \end{pmatrix}\\
    &  &  &  & 0
 \end{pmatrix};
\end{equation}
and
   \begin{equation}
{\rm Pf} (\frac{R^{E^\perp}-L_{2}}{2})=\prod_{ j=1 }^{\frac{n-k-1}{2}}\frac{\hat{\theta}_{j}}{2}.
 \end{equation}
Similar to Theorem 3.9, we get the main Theorem in this section.

\begin{thm}( Local odd dimensional  equivariant   index Theorem  for sub-signature operators )

Let $x_0\in M^\gamma$, then
\begin{eqnarray}
&& \lim_{t\rightarrow 0}{\rm Tr}\left[\tilde{\gamma}(x_0)\hat{c}(E,g^{E})I_{(F+\partial_t)^{-1}}(x_0,t)\right]\nonumber\\
&=& -(\frac{1}{\sqrt{-1}})^{\frac{k}{2}-1} 2^{\frac{n}{2}}\Big\{\widehat{A}(R^{M^{\gamma}})\nu_{\phi}(R^{N^{\gamma}})
 i^{*}_{M^{\gamma}}\Big[{\rm det}^{\frac{1}{2} } \Big({\rm cosh}\big( \frac{R^{E}}{4\pi}-\frac{L_{1}}{2}\big) \Big)\nonumber\\
&&\times {\rm det}^{\frac{1}{2} } \Big( \frac{ {\rm sinh}
(\frac{R^{E^{\perp}}}{4\pi}-\frac{L_{2}}{2})}{\frac{R^{E^{\perp}}}{4\pi}-\frac{L_{2}}{2}}\Big)  {\rm Pf} \Big(\frac{R^{E^{\perp}}}{4\pi}-\frac{L_{2}}{2}\Big)\Big]\Big\}
^{(a,0)}(x_{0}).
\end{eqnarray}

\end{thm}

\section*{ Acknowledgements}
This work was supported by NSFC. 11271062 and NCET-13-0721.  The authors also thank
the referee for his (or her) careful reading and helpful comments.


\section*{References}
\begin{thebibliography}{00}
\bibitem{BV1}  N. Berline and M. Vergne, A computation of the equivariant index of the Dirac
operators. Bull. Soc. Math. Prance 113(1985) 305-345.
\bibitem{LYZ}  J. D. Lafferty, Y. L. Yu and W. P. Zhang, A direct geometric proof of Lefschetz
fixed point formulas, Trans. AMS. 329 (1992), 571-583.
\bibitem{PW}  R. Ponge and H. Wang, Noncommutative geometry, conformal geometry, and
the local equivariant index theorem, arXiv:1210.2032.
\bibitem{LM} K. Liu; X. Ma, On family rigidity theorems. I. Duke Math. J. 102(2000), no. 3,
451-474.
\bibitem{Wa2} Y. Wang, The Greiner¡¯s approach of heat kernel asymptotics, equivariant family
JLO characters and equivariant eta forms, arXiv:1304.7354. 1-27.
\bibitem{Wa3} Y. Wang,The Greiner's approach of heat kernel asymptotics and the variation formulas for the equivariant Ray-Singer metric
arXiv:1311.7527.
\bibitem{DF} D. Freed, Two index theorems in odd dimensions, Commu. Anal. Geom. 6(1998), 317-329.
\bibitem{Wa1} Y. Wang, Chern-Connes character for the invariant Dirac operator in odd dimensions.
Sci. China Ser. A 48 (2005), no. 8, 1124-1134.
\bibitem{LW} KF, Liu; Y. Wang, Rigidity Theorems on Odd Dimensional Manifolds. Pure and Applied Mathematics Quarterly.
5 (2009), 1139-1159.
\bibitem{Zh} W. Zhang, Sub-signature operators, $\eta$-invariants and a Riemann-Roch theorem for flat vector bundles.
Chin. Ann. Math. 25B, 7-36 (2004).
\bibitem{Zh1} W. Zhang, Sub-signature operator and its local index theorem. Chinese Sci. Bull. 41, 294-295 (1996). (in Chinese)
 \bibitem{MZ} X. Ma,  W.Zhang.: Eta-invariants, torsion forms and flat vector bundles. Math. Ann. 340: 569-624(2008)
 \bibitem{DZ} X. Dai,  W.Zhang.: Adiabatic limit, Bismut-Freed connection, and the real analytic torsion form.
J. reine angew. Math. 647, 87-113(2010).
\bibitem{BeGS} R. Beals, P. Greiner, N. Stanton, The heat equation on a CR manifold. J.
Differential Geom. 20, 343-387(1984).
\bibitem{ZJW}  J. W. Zhou, A geometric proof of the Lefschetz fixed-point theorem for signature
operators,(in chinese) Acta Math. Sinica 35 (1992), no. 2, 230-239.
\bibitem{Gr} P. Greiner, An asymptotic expansion for the heat equation. Arch. Rational Mech.
Anal. 41(1971), 163-218.

\bibitem{BGV} N. Berline; E. Getzler; M. Vergne, Heat kernels and Dirac operators. Springer-
Verlag, Berlin, 1992.

\bibitem{Po} R. Ponge, A new short proof of the local index formula and some of its applications. Comm. Math. Phys. 241(2003), 215-234.
\bibitem{JMB} J. M. Bismut, The Atiyah-Singer index theorem for families of Dirac operators: Two heat equation proofs.
Invent. math. 83(1986), 91-151.









\end{thebibliography}
\end{document}